\documentclass[10pt]{extarticle}

\usepackage[english]{babel}
\usepackage{graphicx}
\usepackage{framed}
\usepackage[normalem]{ulem}
\usepackage{indentfirst,ragged2e}
\usepackage{amsmath,amsthm,amssymb,amsfonts,wasysym,verbatim,bbm}
\usepackage{mathtools}
\usepackage{esint}
\usepackage{cancel}
\usepackage[italicdiff]{physics}
\usepackage[T1]{fontenc}
\usepackage{lmodern,mathrsfs}
\usepackage[dvipsnames]{xcolor}
\usepackage{nicematrix}
\usepackage{wrapfig}
\usepackage[inline,shortlabels]{enumitem}
\setlist{topsep=2pt,itemsep=2pt,parsep=0pt,partopsep=0pt}
\usepackage[utf8]{inputenc}
\usepackage[a4paper,top=0.7in,bottom=0.7in,left=0.5in,right=0.5in]{geometry}
\usepackage{multicol}
\usepackage{diagbox}
\usepackage[most]{tcolorbox}
\usepackage{tikz,tikz-3dplot,tikz-cd,tkz-tab,tkz-euclide,tikzsymbols,pgf,pgfplots}
\pgfplotsset{compat=newest}
\usepgfplotslibrary{fillbetween}
\usepackage{subfiles}
\graphicspath{{images/}{../images/}} 
\usepackage[backend=bibtex,style=numeric]{biblatex}
\usepackage{csquotes} 
\usepackage[colorlinks,linkcolor=.,citecolor=ForestGreen,urlcolor=RedViolet]{hyperref}
\usepackage[nameinlink]{cleveref}

\newcommand{\hide}[1]{} 

\usepackage[showisoZ=false]{datetime2} 

\DTMnewdatestyle{newdate}{

}

\newcommand*{\utcpm}[1]{
\ifboolexpe{test{\ifnumequal{#1}{0}}}{}{\ifnum#1<0\else+\fi{#1}}}
\DTMnewzonestyle{newzone}{
    
}
\AtBeginDocument{
    \DTMsetdatestyle{newdate}
    \DTMsetzonestyle{newzone}
}

\usepackage{fancyhdr}
\DeclareDocumentCommand\ip{ l m }{\braces#1{\langle}{\rangle}{#2}} 
\DeclareDocumentCommand\floor{ l m }{\braces#1{\lfloor}{\rfloor}{#2}} 
\DeclareDocumentCommand\ceil{ l m }{\braces#1{\lceil}{\rceil}{#2}} 

\docsvlist{A,B,C,D,F,G,I,J,K,M,N,Q,R,T,U,V,W,X,Y,Z}

\docsvlist{A,B,C,D,E,F,G,H,I,J,K,L,M,N,O,P,Q,R,S,T,U,V,W,X,Y,Z}

\docsvlist{a,b,c,d,e,f,g,i,j,k,l,m,n,o,p,q,r,s,t,u,v,w,x,y,z,A,B,C,D,E,F,G,H,I,J,K,L,M,N,O,P,Q,R,S,T,U,V,W,X,Y,Z}

\docsvlist{a,b,c,d,e,f,g,h,i,j,k,l,m,n,o,q,r,t,u,w,x,y,z,A,B,C,D,E,F,G,H,I,J,K,L,M,N,O,P,Q,R,S,T,U,V,W,X,Y,Z}

\docsvlist{a,b,c,d,e,f,g,h,i,j,k,l,m,n,o,p,q,r,s,t,u,v,w,x,y,z,A,B,C,D,E,F,G,H,I,J,K,L,M,N,O,P,Q,R,S,T,U,V,W,X,Y,Z}

\newcommand{\limk}{\lim_{k\to\infty}}

\newcommand{\limn}{\lim_{n\to\infty}}

\newcommand{\sumn}[1][1]{\sum_{n=#1}^\infty}

\newcommand*{\frontone}[1]{\textcolor{red!80!black}{#1}}
\newcommand*{\fronttwo}[1]{\textcolor{Green}{#1}}
\newcommand*{\frontend}[1]{\textcolor{blue!80!black}{#1}}

\newsavebox{\imaginarybox}

\newtheoremstyle{mytheoremstyle}{}{}{}{}{\sffamily\bfseries}{.}{ }{}
\newtheoremstyle{myconjecturestyle}{}{}{}{}{\sffamily\bfseries}{.}{ }{\thmnote{#3}}
\makeatletter
\renewenvironment{proof}[1][\proofname] {\par\pushQED{\qed}{\normalfont\sffamily\bfseries\topsep6\p@\@plus6\p@\relax #1\@addpunct{.} }}{\popQED\endtrivlist\@endpefalse}
\makeatother

\theoremstyle{mytheoremstyle}{\newtheorem{definition}{Definition}[section]}
\theoremstyle{mytheoremstyle}{\newtheorem{proposition}[definition]{Proposition}}
\theoremstyle{mytheoremstyle}{\newtheorem{theorem}[definition]{Theorem}}
\theoremstyle{mytheoremstyle}{\newtheorem{lemma}[definition]{Lemma}}
\theoremstyle{mytheoremstyle}{\newtheorem{corollary}[definition]{Corollary}}
\theoremstyle{mytheoremstyle}{\newtheorem*{remark}{Remark}}
\theoremstyle{mytheoremstyle}{\newtheorem*{remarks}{Remarks}}
\theoremstyle{mytheoremstyle}{\newtheorem*{example}{Example}}
\theoremstyle{mytheoremstyle}{}
\theoremstyle{myconjecturestyle}{}

\definecolor{tcol_DEF}{HTML}{E40125} 
\definecolor{tcol_PRP}{HTML}{EB8407} 
\definecolor{tcol_LEM}{HTML}{05C4D9} 
\definecolor{tcol_THM}{HTML}{1346E4} 
\definecolor{tcol_COR}{HTML}{7904C2} 
\definecolor{tcol_EXA}{HTML}{21340A} 
\definecolor{tcol_REM}{HTML}{18B640} 
\definecolor{tcol_PRF}{HTML}{5A76B2} 

\tcbset{
    tbox_DPLTC_style/.style={
        enhanced jigsaw,
        colback=#1!10, colframe=#1!80!black,
        boxrule=0pt,
        fonttitle=\sffamily\bfseries, coltitle=black,
        separator sign={}, label separator={},
        description font=\normalfont\sffamily,
        description delimiters={(}{)},
        attach title to upper, after title={.\ },
        frame hidden, borderline west={2pt}{0pt}{#1},
        sharp corners,
        top=2pt, bottom=2pt, left=5pt, right=5pt,
        beforeafter skip=10pt, breakable,
    },
    tbox_DPLTC_style*/.style={
        tbox_DPLTC_style=#1,
        interior hidden,
        top=-2pt, bottom=-2pt,
    },
}

\tcolorboxenvironment{definition}{tbox_DPLTC_style=tcol_DEF}
\tcolorboxenvironment{proposition}{tbox_DPLTC_style=tcol_PRP}
\tcolorboxenvironment{theorem}{tbox_DPLTC_style=tcol_THM}
\tcolorboxenvironment{lemma}{tbox_DPLTC_style=tcol_LEM}
\tcolorboxenvironment{corollary}{tbox_DPLTC_style=tcol_COR}
\tcolorboxenvironment{proof}{boxrule=0pt,boxsep=0pt,blanker,borderline west={2pt}{0pt}{CadetBlue!80!white},left=8pt,right=8pt,sharp corners,before skip=10pt,after skip=10pt,breakable}
\tcolorboxenvironment{remark}{tbox_DPLTC_style*=tcol_REM}
\tcolorboxenvironment{remarks}{tbox_DPLTC_style*=tcol_REM}
\tcolorboxenvironment{example}{tbox_DPLTC_style*=tcol_EXA}
\tcolorboxenvironment{examples}{tbox_DPLTC_style*=tcol_EXA}
\tcolorboxenvironment{cthm}{boxrule=0pt,boxsep=0pt,colback={gray!10},left=8pt,right=8pt,enhanced jigsaw, borderline west={2pt}{0pt}{gray},sharp corners,before skip=10pt,after skip=10pt,breakable}

\AddToHook{env/proposition/begin}{\crefalias{definition}{proposition}}
\AddToHook{env/lemma/begin}{\crefalias{definition}{lemma}}
\AddToHook{env/theorem/begin}{\crefalias{definition}{theorem}}
\AddToHook{env/corollary/begin}{\crefalias{definition}{corollary}}

\crefname{definition}{definition}{definitions}
\crefformat{definition}{\normalfont\sffamily\bfseries\color{tcol_DEF} #2definition~#1#3}
\crefrangeformat{definition}{\normalfont\sffamily\bfseries\color{tcol_DEF} #3definitions~#1#4 {\normalfont\color{black}and} #5#2#6}
\Crefformat{definition}{\normalfont\sffamily\bfseries\color{tcol_DEF} #2Definition~#1#3}
\Crefrangeformat{definition}{\normalfont\sffamily\bfseries\color{tcol_DEF} #3Definitions~#1#4 {\normalfont\color{black}and} #5#2#6}

\crefname{proposition}{proposition}{propositions}
\crefformat{proposition}{\normalfont\sffamily\bfseries\color{tcol_PRP} #2proposition~#1#3}
\crefrangeformat{proposition}{\normalfont\sffamily\bfseries\color{tcol_PRP} #3propositions~#1#4 {\normalfont\color{black}and} #5#2#6}
\Crefformat{proposition}{\normalfont\sffamily\bfseries\color{tcol_PRP} #2Proposition~#1#3}
\Crefrangeformat{proposition}{\normalfont\sffamily\bfseries\color{tcol_PRP} #3Propositions~#1#4 {\normalfont\color{black}and} #5#2#6}

\crefname{lemma}{lemma}{lemmas}
\crefformat{lemma}{\normalfont\sffamily\bfseries\color{tcol_LEM} #2lemma~#1#3}
\crefrangeformat{lemma}{\normalfont\sffamily\bfseries\color{tcol_LEM} #3lemmas~#1#4 {\normalfont\color{black}and} #5#2#6}
\Crefformat{lemma}{\normalfont\sffamily\bfseries\color{tcol_LEM} #2Lemma~#1#3}
\Crefrangeformat{lemma}{\normalfont\sffamily\bfseries\color{tcol_LEM} #3Lemmas~#1#4 {\normalfont\color{black}and} #5#2#6}

\crefname{theorem}{theorem}{theorems}
\crefformat{theorem}{\normalfont\sffamily\bfseries\color{tcol_THM} #2theorem~#1#3}
\crefrangeformat{theorem}{\normalfont\sffamily\bfseries\color{tcol_THM} #3theorems~#1#4 {\normalfont\color{black}and} #5#2#6}
\Crefformat{theorem}{\normalfont\sffamily\bfseries\color{tcol_THM} #2Theorem~#1#3}
\Crefrangeformat{theorem}{\normalfont\sffamily\bfseries\color{tcol_THM} #3Theorems~#1#4 {\normalfont\color{black}and} #5#2#6}

\crefname{corollary}{corollary}{corollaries}
\crefformat{corollary}{\normalfont\sffamily\bfseries\color{tcol_COR} #2corollary~#1#3}
\crefrangeformat{corollary}{\normalfont\sffamily\bfseries\color{tcol_COR} #3corollaries~#1#4 {\normalfont\color{black}and} #5#2#6}
\Crefformat{corollary}{\normalfont\sffamily\bfseries\color{tcol_COR} #2Corollary~#1#3}
\Crefrangeformat{corollary}{\normalfont\sffamily\bfseries\color{tcol_COR} #3Corollaries~#1#4 {\normalfont\color{black}and} #5#2#6}

\newenvironment{talign*}{\let\displaystyle\textstyle\csname align*\endcsname}{\endalign}

\makeatletter
\def\fsize{\dimexpr\f@size pt\relax}
\makeatother

\usepackage[explicit]{titlesec}
\titleformat{\section}{\normalfont\sffamily\Large\bfseries}{\thesection}{12pt}{#1}
\titleformat{\subsection}{\normalfont\sffamily\large\bfseries}{\thesubsection}{12pt}{#1}
\titleformat{\subsubsection}{\normalfont\sffamily\large\bfseries}{\thesubsection}{8pt}{#1}

\titlespacing*{\section}{0pt}{5pt}{5pt}
\titlespacing*{\subsection}{0pt}{5pt}{5pt}
\titlespacing*{\subsubsection}{0pt}{5pt}{5pt}

\newcommand{\Disp}{\displaystyle}

\DeclareMathAlphabet\mathbfcal{OMS}{cmsy}{b}{n}
\makeatletter
\g@addto@macro\normalsize{
\setlength\abovedisplayskip{3pt}
\setlength\belowdisplayskip{3pt}
\setlength\abovedisplayshortskip{0pt}
\setlength\belowdisplayshortskip{0pt}}
\makeatother

\makeatletter
\renewcommand\maketitle{
    \null\vspace{6mm}
    \begin{center}
        {\Huge\sffamily\bfseries\selectfont\@title}\\
            \vspace{6mm}
        {\Large\sffamily\selectfont\@author}\\
            \vspace{6mm}
        {\large\sffamily\selectfont\@date}
    \end{center}
    \vspace{6mm}
}
\makeatother

\definecolor{abstract_color}{HTML}{0A0848}
\renewenvironment{abstract}{
\begin{tcolorbox}[
    enhanced, width=6in, center, title={Abstract},
    frame hidden, colback=white, colbacktitle=white,
    fonttitle=\sffamily\bfseries, coltitle=black,
    borderline north={1.5pt}{0pt}{abstract_color},
    attach boxed title to top center={
        yshift=-0.25mm-\tcboxedtitleheight/2, yshifttext=2mm-\tcboxedtitleheight/2
    },
    boxed title style={
        boxrule=0.5mm,
        frame code={
            \path[tcb fill frame, abstract_color] ([xshift=-4mm]frame.west)--(frame.north west)--(frame.north east)--([xshift=4mm]frame.east)--(frame.south east)--(frame.south west)--cycle;
        },
        interior code={
            \path[tcb fill interior] ([xshift=-2mm]interior.west)--(interior.north west)--(interior.north east)--([xshift=2mm]interior.east)--(interior.south east)--(interior.south west)--cycle;
        }
    },
    top=\tcboxedtitleheight/2, bottom=0pt, left=0pt, right=0pt]
    }
{\end{tcolorbox}}

\title{Counting words without non-decreasing subwords of fixed length}
\author{Senan Sekhon}
\date{\DTMToday}

\begin{document}

\thispagestyle{empty}

\maketitle

\tableofcontents

\begin{abstract}
    In \cite{sekhon}, we derived exact formulas for generating functions counting the number of $n$-ary words avoiding \textit{strictly} increasing subwords of length $k$, and provided applications in probability theory as well as the continuous limit as $n\to\infty$. We also conjectured several corresponding formulas for the case where the ``strictly'' requirement is dropped. In this paper, we prove those formulas.
\end{abstract}


\vspace{10pt}

\begin{tcolorbox}[
    colframe=black,
    title={Summary of Results},
    fonttitle=\large\sffamily\bfseries,
    center title,
]
    \subsubsection*{Terminology}
    
    Define the following:
    \begin{itemize}
        \item An \textbf{$n$-ary word} of \textbf{length} $s$ is a tuple $(a_1,a_2,\ldots,a_s)$ of $s$ elements chosen from $\{1,2,\ldots,n\}$.
        \item A \textbf{subword} of a word $(a_1,a_2,\ldots,a_s)$ is a contiguous tuple of the form $(a_i,a_{i+1},\ldots,a_j)$ where $1\le i<j\le s$.
        \item A \textbf{soft streak of length $k$} in a word is a non-deincreasing subword of length $k$.
        \item $\Ps(n,k,s)$ is the number of $n$-ary words of length $s$ that do not contain any soft streak of length $k$.
        \item $f_{n,k}(z)$ is the generating function of $\Ps(n,k,s)$, i.e. $f_{n,k}(z)=\sum_{s=0}^\infty \Ps(n,k,s)z^s$.
    \end{itemize}

    \vspace{10pt}
    
    Then $f_{n,k}(z)$ is given by:
    \begin{equation*}
        \tcbhighmath[
            colback=white, colframe=JungleGreen,
        ]{
            f_{n,k}(z)
            = \frac{\qty(1-z^k)^n}{\Disp\sum_{r=0}^{(k-1)n} \psi_{k,r}\Bs(n,k,r)z^r}
            = \frac{k}{\Disp\sum_{r=1}^{k-1} \qty(1-\omega_k^{-r})\qty(1-\omega_k^r z)^{-n}}
        }
    \end{equation*}
    
    Where:
    \begin{align*}
        \omega_k = e^{\tfrac{2\pi i}{k}}
        &&
        \psi_{k,r} = \begin{cases}
            1 & r \equiv 0 \pmod k \\
            -1 & r \equiv 1 \pmod k \\
            0 & \text{otherwise}
        \end{cases}
        &&
        \parbox{12em}{\centering$\Bs(n,k,r)$ is the coefficient of $x^r$ in the expansion of $\qty(1+x+x^2+\cdots+x^{k-1})^n$}
    \end{align*}
\end{tcolorbox}

\newpage

\section{Preliminary Definitions}

In this paper:
\begin{itemize}
    \item $\N$ denotes the set $\{1,2,3,\ldots\}$, while $\N_0$ denotes the set $\{0,1,2,3,\ldots\}$.
    \item $\binom{n}{k}$ denotes the binomial coefficient $\frac{n!}{(n-k)!k!}$, which is defined for all $n\in\N_0$ and all $k\in\Z$. If $k<0$ or $k>n$, we have $\binom{n}{k}=0$.
    \item $\omega_k$ denotes $e^{\tfrac{2\pi i}{k}}$, the $k$\textsuperscript{th} root of unity with the least positive argument.
\end{itemize}
Also, for $k\ge 2$ and $r\in\Z$:
\begin{itemize}
    \item $\psi_{k,r}$ denotes $1$ if $r\equiv 0\pmod{k}$, $1$ if $r\equiv 0\pmod{k}$, and $0$ otherwise. \hfill (\cite[Definition 1.1]{sekhon})
    \item $\Bs(n,k,r)$ denotes the coefficient of $x^r$ in the expansion of $\qty(1+x+x^2+\cdots+x^{k-1})^n$. \hfill (\cite[Definition 4.2]{sekhon})
\end{itemize}

\vspace{5pt}

The following identity appears as Corollary 4.4 in \cite{sekhon}. It does not matter if we include the term with $s=0$ (or $s=k$) as this is zero.

\begin{corollary}\label{psi_k_r:identity_2}
    \begin{equation*}
        \sum_{r=0}^{(k-1)n} \psi_{k,r}\Bs(n,k,r)z^r
        = \frac{\qty(1-z^k)^n}{k}\sum_{s=1}^{k-1} \qty(1-\omega_k^{-s})\qty(1-\omega_k^s z)^{-n}
    \end{equation*}
\end{corollary}

\section{A lemma from the Goulden-Jackson method}

In this section, we use the notation of \cite[Definition 2.2]{sekhon}. See \cite[\S2]{sekhon} for a description of the Goulden-Jackson method.

\begin{lemma}\label{goulden_jackson:transformation_lemma}
    Suppose $\vb{a}$ and $\vb{b}$ are two forbidden words of the same length. Define $\mathfrak{A}$ as the set of all forbidden words that front-run $\vb{a}$, and $\mathfrak{B}$ as the set of all forbidden words that front-run $\vb{b}$. Suppose there is a bijection $\phi:\mathfrak{A}\to\mathfrak{B}$ such that for all $\vb{x}\in\mathfrak{A}$, we have $\abs{\vb{b}\wedge\vb{\phi(\vb{x})}}=\abs{\vb{a}\wedge\vb{x}}$ and $W(\phi(\vb{x}))=W(\vb{x})$. Then $W(\vb{a})=W(\vb{b})$.
\end{lemma}
In other words, if we have two forbidden words $\vb{a}$ and $\vb{b}$ of the same length, and every forbidden word that front-runs $\vb{a}$ can be transformed into a forbidden word that front-runs $\vb{b}$ with the \emph{same overlap} and the \emph{same weight} (and vice versa), then $W(\vb{a})=W(\vb{b})$.
\begin{proof}
    By the definition of $W(\vb{x})$, we have:
    \begin{align*}
        W(\vb{a})
        = -z^\abs{\vb{a}} - \sum_{\vb{c}\in\mathfrak{A}} z^{\abs{\vb{a}}-\abs{\vb{a}\wedge\vb{c}}}W(\vb{c})
        &&
        W(\vb{b})
        = -z^\abs{\vb{b}} - \sum_{\vb{d}\in\mathfrak{B}} z^{\abs{\vb{b}}-\abs{\vb{b}\wedge\vb{d}}}W(\vb{d})
    \end{align*}
    By assumption, we have $\abs{\vb{a}}=\abs{\vb{b}}$, and for all $\vb{c}\in\mathfrak{A}$ (setting $\vb{d}=\phi(\vb{c})$), we have $\abs{\vb{a}\wedge\vb{c}}=\abs{\vb{b}\wedge\vb{d}}$ and $W(\vb{c})=W(\vb{d})$. Thus $W(\vb{a})=W(\vb{b})$.
\end{proof}
\begin{remark}
    In practice, when we use this lemma, the situation will be symmetric in $\vb{a}$ and $\vb{b}$, so we only need to show that every forbidden word that front-runs $\vb{a}$ can be transformed into a forbidden word that front-runs $\vb{b}$ with the same overlap and the same weight. The reverse direction will follow by switching $\vb{a}$ and $\vb{b}$.
\end{remark}

\begin{corollary}\label{goulden_jackson:last_letter_corollary}
    Suppose $\Fs$ is a set of forbidden words, all of the same length. Then $W(\vb{a})$ does not depend on the last letter of $\vb{a}$.
\end{corollary}
\begin{proof}
    Suppose all forbidden words have the same length $k$. Suppose $\vb{a}=(a_1,a_2,\ldots,a_k)$ and $\vb{b}=(b_1,b_2,\ldots,b_k)$ are two forbidden words such that $a_i=b_i$ for all $i\in\{1,2,\ldots,k-1\}$ (in other words, $\vb{a}$ and $\vb{b}$ may differ only in their last letter). We want to show that $W(\vb{a})=W(\vb{b})$. Suppose $\vb{c}=(c_1,c_2,\ldots,c_k)$ is a forbidden word that front-runs $\vb{a}$ with an overlap of $j$, where $j\in\{1,2,\ldots,k-1\}$. Then we have $c_{k-j+i}=a_i=b_i$ for all $i\in\{1,2,\ldots,j\}$. Thus $\vb{c}$ also front-runs $\vb{b}$ with an overlap of $j$. Using the notation in \Cref{goulden_jackson:transformation_lemma}, we have $\mathfrak{A}=\mathfrak{B}$ and we can define $\phi(\vb{x})=\vb{x}$. By \Cref{goulden_jackson:transformation_lemma}, we have $W(\vb{a})=W(\vb{b})$.
\end{proof}

\section{Avoiding soft streaks of length $k$}

In this section, we will focus on applying the Goulden-Jackson method to the case where $\Fs$ is the set of all non-decreasing words of length $k$. For example, if $n=4$ and $k=3$, we have:
\begin{align*}
    \Fs = \{&(1,1,1),(1,1,2),(1,1,3),(1,1,4),(1,2,2),(1,2,3),(1,2,4),(1,3,3),(1,3,4),(1,4,4), \\
    &(2,2,2),(2,2,3),(2,2,4),(2,3,3),(2,3,4),(2,4,4),(3,3,3),(3,3,4),(3,4,4),(4,4,4)\}
\end{align*}
In general, we have $\abs{\Fs}=\binom{n+k-1}{k}$.

\begin{definition}
    A \textbf{soft streak} of length $k$ (or more simply a \emph{soft streak of $k$}) is any non-decreasing word of length $k$.
\end{definition}

Similarly to the previous section, we start by proving a lemma akin to \cite[Lemma 3.2]{sekhon}. This time, we no longer have the benefit of performing induction on the first letter (which we could do for \emph{strictly} increasing words), so a different approach is required. We instead perform induction on the \emph{number} of letters of $\vb{a}$ that $W(\vb{a})$ is independent of.

\begin{lemma}\label{soft_streak:w_depends_only_on_1}
    $W(a_1,a_2,\ldots,a_k)$ depends only on $a_1$ and not on $a_2,\ldots,a_k$.
\end{lemma}
\begin{proof}
    Since all of the forbidden words have length $k$, it follows from \Cref{goulden_jackson:last_letter_corollary} that $W(a_1,a_2,\ldots,a_k)$ does not depend on $a_k$. Suppose $W(a_1,a_2,\ldots,a_k)$ does not depend on any of $a_{r+1},a_{r+2},\ldots,a_k$, where $r\in\{2,3,\ldots,k-1\}$. In other words, $W(\vb{a})$ does not depend on the last $k-r$ letters of $\vb{a}$. We will show that it also does not depend on $a_r$.\\
    
    Suppose $\vb{a}=(a_1,a_2,\ldots,a_k)$ and $\vb{b}=(b_1,b_2,\ldots,b_k)$ are two soft streaks of length $k$ such that $a_i=b_i$ for all $i\in\{1,2,\ldots,r-1\}$. We want to show that $W(\vb{a})=W(\vb{b})$. Suppose $\vb{c}=(c_1,c_2,\ldots,c_k)$ is a soft streak of length $k$ that front-runs $\vb{a}$ with an overlap of $j$, where $1\le j\le k-1$. Then we have $c_{k-j+i}=a_i$ for all $i\in\{1,2,\ldots,j\}$. Now define:
    \begin{equation*}
        \vb{d} = (c_1,c_2,\ldots,c_{k-j},b_1,b_2,\ldots,b_j)
    \end{equation*}
    In other words, $\vb{d}$ is a copy of $\vb{c}$ where the part overlapping with $\vb{a}$ is replaced with the corresponding elements of $\vb{b}$. Thus $\vb{d}$ front-runs $\vb{b}$ with an overlap of $j$. Since $\vb{c}$ is non-decreasing, we have $c_1\le c_2\le\cdots\le c_{k-j+1}$. Since $\vb{b}$ is non-decreasing, we have $b_1\le b_2\le\cdots\le b_j$. We also have $c_{k-j+1}=a_1=b_1$. This yields:
    \begin{equation*}
        c_1\le c_2\le\cdots\le c_{k-j}\le b_1\le b_2\le\cdots\le b_j
    \end{equation*}
    Thus $\vb{d}$ is non-decreasing, and so it is a soft streak. Switching $\vb{a}$ and $\vb{b}$ in the above argument allows us to recover $\vb{c}$ from $\vb{d}$, with the same overlap and the same weight. Therefore, by \Cref{goulden_jackson:transformation_lemma}, we have $W(\vb{a})=W(\vb{b})$. Thus $W(a_1,a_2,\ldots,a_k)$ does not depend on $a_r$, and so it does not depend on any of $a_2,\ldots,a_k$.
\end{proof}
\begin{remarks}\leavevmode
    \begin{enumerate}
        \item It may seem as if the above argument holds for any set $\Fs$ of forbidden words that all have length $k$. However, we used the fact that $\Fs$ contains \emph{all} the soft streaks to conclude that $\vb{d}$ is indeed a forbidden word. We also need the fact that $\Fs$ contains \emph{only} the soft streaks to be able to recover $\vb{c}$ from $\vb{d}$.
        \item The induction step fails when $r=1$ as we would no longer have $c_{k-j+1}=a_1=b_1$, which we needed to conclude that $\vb{d}$ is a soft streak. And as we will see shortly, $W(a_1,a_2,\ldots,a_k)$ \emph{does} depend on $a_1$.
    \end{enumerate}
\end{remarks}

\begin{example}
    Suppose $k=4$ (and $n$ remains arbitrary). We can write out the distinct weights as follows:
    \begin{align*}
        W(1,2,3,4)
        &= -z^4 - z^3W(1,1,1,1) - z^2W(1,1,1,2) - zW(1,1,2,3) \\
        W(2,3,4,5)
        &= -z^4 - z^3W(1,1,1,2) - z^3W(1,1,2,2) - z^3W(1,2,2,2) - z^3W(2,2,2,2) \\
        &\qquad\quad - z^2W(1,1,2,3) - z^2W(1,2,2,3) - z^2W(2,2,2,3) \\
        &\qquad\quad - zW(1,2,3,4) - zW(2,2,3,4) \\
        W(3,4,5,6)
        &= -z^4 - z^3W(1,1,1,3) - z^3W(1,1,2,3) - z^3W(1,1,3,3) - z^3W(1,2,2,3) - z^3W(1,2,3,3) \\
        &\qquad\quad - z^3W(1,3,3,3) - z^3W(2,2,2,3) - z^3W(2,2,3,3) - z^3W(2,3,3,3) - z^3W(3,3,3,3) \\
        &\qquad\quad - z^2W(1,1,3,4) - z^2W(1,2,3,4) - z^2W(1,3,3,4) \\
        &\qquad\quad - z^2W(2,2,3,4) - z^2W(2,3,3,4) - z^2W(3,3,3,4) \\
        &\qquad\quad - zW(1,3,4,5) - zW(2,3,4,5) - zW(3,3,4,5) \\
        W(4,5,6,7)
        &= -z^4 - z^3W(1,1,1,4) - z^3W(1,1,2,4) - z^3W(1,1,3,4) - z^3W(1,1,4,4) - z^3W(1,2,2,4) \\
        &\qquad\quad - z^3W(1,2,3,4) - z^3W(1,2,4,4) - z^3W(1,3,3,4) - z^3W(1,3,4,4) - z^3W(1,4,4,4) \\
        &\qquad\quad - z^3W(2,2,2,4) - z^3W(2,2,3,4) - z^3W(2,2,4,4) - z^3W(2,3,3,4) - z^3W(2,3,4,4) \\
        &\qquad\quad - z^3W(2,4,4,4) - z^3W(3,3,3,4) - z^3W(3,3,4,4) - z^3W(3,4,4,4) - z^3W(4,4,4,4) \\
        &\qquad\quad - z^2W(1,1,4,5) - z^2W(1,2,4,5) - z^2W(1,3,4,5) - z^2W(1,4,4,5) - z^2W(2,2,4,5) \\
        &\qquad\quad - z^2W(2,3,4,5) - z^2W(2,4,4,5) - z^W(3,3,4,5) - z^2W(3,4,4,5) - z^2W(4,4,4,5) \\
        &\qquad\quad - zW(1,4,5,6) - zW(2,4,5,6) - zW(3,4,5,6) - zW(4,4,5,6)
    \end{align*}
    Using \Cref{soft_streak:w_depends_only_on_1}, we can simplify this to:
    \begin{align*}
        w_1
        &= -z^4 - z^3w_1 - z^2w_1 - zw_1 \\
        &= -z^4 - (z^3+z^2+z)w_1 \\
        w_2
        &= -z^4 - z^3w_1 - z^3w_1 - z^3w_1 - z^3w_2 - z^2w_1 - z^2w_1 - z^2w_2 - zw_1 - zw_2 \\
        &= -z^4 - (3z^3+2z^2+z)w_1 - (z^3+z^2+z)w_2 \\
        w_3
        &= -z^4 - z^3w_1 - z^3w_1 - z^3w_1 - z^3w_1 - z^3w_1 - z^3w_1 - z^3w_2 - z^3w_2 - z^3w_2 - z^3w_3 \\
        &\qquad\quad - z^2w_1 - z^2w_1 - z^2w_1 - z^2w_2 - z^2w_2 - z^2w_3 - zw_1 - zw_2 - zw_3 \\
        &= -z^4 - (6z^3+3z^2+z)w_1 - (3z^3+2z^2+z)w_2 - (z^3+z^2+z)w_3 \\
        w_4
        &= -z^4 - z^3w_1 - z^3w_1 - z^3w_1 - z^3w_1 - z^3w_1 - z^3w_1 - z^3w_1 - z^3w_1 - z^3w_1 - z^3w_1 \\
        &\qquad\quad - z^3w_2 - z^3w_2 - z^3w_2 - z^3w_2 - z^3w_2 - z^3w_2 - z^3w_3 - z^3w_3 - z^3w_3 - z^3w_4 \\
        &\qquad\quad - z^2w_1 - z^2w_1 - z^2w_1 - z^2w_1 - z^2w_2 \\
        &\qquad\quad - z^2w_2 - z^2w_2 - z^2w_3 - z^2w_3 - z^2w_4 \\
        &\qquad\quad - zw_1 - zw_2 - zw_3 - zw_4 \\
        &= -z^4 - (10z^3+4z^2+z)w_1 - (6z^3+3z^2+z)w_2 - (3z^3+2z^2+z)w_3 - (z^3+z^2+z)w_4
    \end{align*}
    More generally, for all $m\in\N_0$:
    \begin{talign*}
        w_m
        &= -z^4 - \qty(\binom{m+1}{2}z^3+mz^2+z)w_1 - \qty(\binom{m}{2}z^3+(m-1)z^2+z)w_2 - \cdots \\
        &\qquad\quad - \qty(6z^3+3z^2+z)w_{m-2} - \qty(3z^3+2z^2+z)w_{m-1} - (z^3+z^2+z)w_m
    \end{talign*}
    Similarly to the previous section, these expressions do \emph{not} depend on $n$. However, this time, $w_m$ appears on the right side of the equation for $w_m$, and so it is not an explicit equation. Of course, we can easily make it explicit, but we will defer this step to simplify some of our intermediate work.
\end{example}

\begin{proposition}\label{soft_streak:weights_recursive_formula}
    For all $m\in\N$, we have:
    \begin{equation*}
        w_m = -z^k - \sum_{r=1}^m \sum_{s=1}^{k-1} \binom{m-r+s-1}{s-1} z^s w_r
    \end{equation*}
\end{proposition}
\begin{proof}
    We start by expanding $w_m$ as follows:
    \begin{align*}
        w_m
        &= W(m,m+1,\ldots,m+k-1) \\
        &= -z^k \\
        &\quad -z^{k-1}\qty(W(1,1,\ldots,1,\frontone{m}) + \cdots + W(m,m,\ldots,m,\frontone{m})) \tag{$\binom{m+k-2}{k-1}$ terms} \\
        &\quad -z^{k-2}\qty(W(1,1,\ldots,1,\fronttwo{m,m+1}) + \cdots + W(m,m,\ldots,m,\fronttwo{m,m+1})) \tag{$\binom{m+k-3}{k-2}$ terms} \\
        &\quad\;\: \vdots \\
        &\quad -z^2\qty(W(1,1,\frontend{m,m+1,\ldots,m+k-3}) + \cdots + W(m,m,\frontend{m,m+1,\ldots,m+k-3})) \tag{$\binom{m+1}{2}$ terms} \\
        &\quad -z\qty(W(1,\frontend{m,m+1,\ldots,m+k-2}) + \cdots + W(m,\frontend{m,m+1,\ldots,m+k-2})) \tag{$m$ terms}
    \end{align*}
    More specifically, the terms with $z^{k-1}$ are the weights of all soft streaks of length $k$ that \frontone{end with $m$}. The terms with $z^{k-2}$ are the weights of all soft streaks of length $k$ that \fronttwo{end with $(m,m+1)$}. Likewise for the \frontend{other terms}.\\
    
    We now simplify the weights on the right using \Cref{soft_streak:w_depends_only_on_1}. To do this, note that among the $\binom{m+k-2}{k-1}$ terms above with $z^{k-1}$, there are $\binom{m+k-3}{k-2}$ terms that start with $1$ (because they are of the form $(1,\xleftrightarrow{k-2},m)$, and the $k-2$ numbers in between are chosen with replacement from $\{1,2,\ldots,m\}$), $\binom{m+k-4}{k-3}$ terms that start with $1$, and so on. Likewise for the terms with other powers of $z$. All in all, we have:
    \begin{align*}
        w_m
        &= -z^k \\
        &\quad - z^{k-1}\qty(\binom{m+k-3}{k-2}w_1 + \binom{m+k-4}{k-2}w_2 + \cdots + \binom{k-1}{k-2}w_{m-1} + \binom{k-2}{k-2}w_m) \\
        &\quad - z^{k-2}\qty(\binom{m+k-4}{k-3}w_1 + \binom{m+k-5}{k-3}w_2 + \cdots + \binom{k-2}{k-3}w_{m-1} + \binom{k-3}{k-3}w_m) \\
        &\quad\;\: \vdots \\
        &\quad - z^2\qty(mw_1 + (m-1)w_2 + \cdots + 2w_{m-1} + w_m) \\
        &\quad - z\qty(w_1 + w_2 + \cdots + w_{m-1} + w_m)
    \end{align*}
    We now regroup the terms to separate $w_1,w_2,\ldots,w_m$:
    \begin{align*}
        w_m
        &= -z^k \\
        &\quad - \qty(\binom{m+k-3}{k-2}z^{k-1} + \binom{m+k-4}{k-3}z^{k-2} + \cdots + mz^2 + z)w_1 \\
        &\quad - \qty(\binom{m+k-4}{k-2}z^{k-1} + \binom{m+k-5}{k-3}z^{k-2} + \cdots + (m-1)z^2 + z)w_2 \\
        &\quad\;\: \vdots \\
        &\quad - \qty((k-1)z^{k-1} + (k-2)z^{k-2} + \cdots + 2z^2 + z)w_{m-1} \\
        &\quad - \qty(z^{k-1} + z^{k-2} + \cdots + z^2 + z)w_{m-1} \\
        &= -z^k - \sum_{r=1}^m \sum_{s=1}^{k-1} \binom{m-r+s-1}{s-1} z^s w_r
        \qedhere
    \end{align*}
\end{proof}

\begin{lemma}\label{soft_streak:rational_function_coefficient_lemma}
    Consider the following expression as a rational function of $x$:
    \begin{equation*}
        \sum_{s=1}^{k-1} \frac{\qty(1-\omega_k^{-s})\qty(1-\omega_k^sz)}{1-x-\omega_k^sz}
    \end{equation*}
    Expanded as a power series in $x$, the coefficient of $x^n$ is:
    \begin{equation*}
        \sum_{s=1}^{k-1} \qty(1-\omega_k^{-s})\qty(1-\omega_k^sz)^{-n}
    \end{equation*}
\end{lemma}
\begin{proof}
    For each $s\in\{1,2,\ldots,k-1\}$, we have:
    \begin{align*}
        \frac{\qty(1-\omega_k^{-s})\qty(1-\omega_k^sz)}{1-x-\omega_k^sz}
        &= \frac{\qty(1-\omega_k^{-s})}{1-\qty(\frac{x}{1-\omega_k^sz})} \\
        &= \qty(1-\omega_k^{-s})\sum_{n=0}^\infty \qty(\frac{x}{1-\omega_k^sz})^n \tag{geometric series in $\frac{x}{1-\omega_k^sz}$} \\
        &= \sum_{n=0}^\infty \qty(1-\omega_k^{-s})\qty(1-\omega_k^sz)^{-n} x^n
    \end{align*}
    The result follows by summing over $s\in\{1,2,\ldots,k-1\}$.
\end{proof}

\begin{proposition}\label{soft_streak:weights_formula_using_generating_function}
    For all $m\in\N$, we have:
    \begin{equation}
        w_m = -\frac{z^k}{k}\sum_{s=1}^{k-1} \qty(1-\omega_k^{-s})\qty(1-\omega_k^sz)^{-m}
    \end{equation}
\end{proposition}
\begin{proof}
    By \Cref{soft_streak:weights_recursive_formula}, we have:
    \begin{align*}
        w_m + \sum_{r=1}^m c_{m-r}w_r = -z^k
        &&
        c_l = \sum_{s=1}^{k-1} \binom{l+s-1}{s-1}z^s
    \end{align*}
    Define the generating functions $\Phi(x)$ and $C(x)$ as follows:
    \begin{align*}
        \Phi(x) = \sum_{m=1}^\infty w_mx^m
        &&
        C(x) = \sum_{l=0}^\infty c_lx^l
    \end{align*}
    We first derive an explicit expression for $C(x)$:
    \begin{align*}
        C(x)
        &= \sum_{l=0}^\infty \sum_{s=1}^{k-1} \binom{l+s-1}{s-1}z^sx^l \\
        &= \sum_{s=1}^{k-1} z^s\sum_{l=0}^\infty \binom{l+s-1}{s-1}x^l \\
        &= \sum_{s=1}^{k-1} z^s(1-x)^{-s} \\
        &= \sum_{s=1}^{k-1} \qty(\frac{z}{1-x})^s \\
        &= \frac{y-y^k}{1-y} \tag{where $y=\frac{z}{1-x}$}
    \end{align*}
    Thus:
    \begin{equation}\label{soft_streak:weights_formula_using_generating_function:1+cx}
        1 + C(x)
        = 1 + \frac{y-y^k}{1-y}
        = \frac{1-y^k}{1-y}
    \end{equation}
    We now show that $\Phi(x)(1+C(x))=-\frac{z^kx}{1-x}$:
    \begin{align*}
        \Phi(x)\qty(1+C(x))
        &= \qty(\sum_{m=1}^\infty w_mx^m)\qty(1+\sum_{l=0}^\infty c_lx^l) \\
        &= \sum_{m=1}^\infty \qty(1+\sum_{l=0}^\infty c_lx^l)w_mx^m \\
        &= \sum_{m=1}^\infty w_mx^m + \sum_{m=1}^\infty \sum_{l=0}^\infty c_lw_mx^{m+l} \\
        &= \sum_{m=1}^\infty w_mx^m + \sum_{m=1}^\infty \sum_{q=m}^\infty c_{q-m}w_mx^q \tag{$q=m+j$} \\
        &= \sum_{m=1}^\infty w_mx^m + \sum_{q=1}^\infty \sum_{m=1}^q c_{q-m}w_mx^q \tag{swap order of sums} \\
        &= \sum_{q=1}^\infty w_qx^q + \sum_{q=1}^\infty \sum_{m=1}^q c_{q-m}w_mx^q \tag{replace $m$ with $q$ in the first term} \\
        &= \sum_{q=1}^\infty \qty(w_q + \sum_{m=1}^q c_{q-m}w_m)x^q \\
        &= \sum_{q=1}^\infty -z^kx^q \\
        &= -\frac{z^kx}{1-x}
    \end{align*}
    Using \eqref{soft_streak:weights_formula_using_generating_function:1+cx} and solving for $\Phi(x)$, we get:
    \begin{align*}
        \Phi(x)
        &= \frac{-z^kx}{1-x}\frac{1-y}{1-y^k} \\
        &= \frac{-z^kx}{1-x}\frac{1-\frac{z}{1-x}}{1-\qty(\frac{z}{1-x})^k} \\
        &= \frac{-z^kx}{1-x}\frac{1-x-z}{1-x}\frac{(1-x)^k}{(1-x)^k-z^k} \\
        &= \frac{-z^kx(1-x-z)(1-x)^{k-2}}{(1-x)^k-z^k}
    \end{align*}
    Since $(1-x)^k-z^k=\prod_{u=0}^{k-1} \qty(1-x-\omega_k^uz)$, we can cancel the factor of $1-x-z$:
    \begin{equation}\label{soft_streak:weights_formula_using_generating_function:eq1}
        \Phi(x) = \frac{-z^kx(1-x)^{k-2}}{\prod_{u=1}^{k-1} \qty(1-x-\omega_k^uz)}
    \end{equation}
    We now treat \eqref{soft_streak:weights_formula_using_generating_function:eq1} as a rational function of $x$ and decompose it into partial fractions:
    \begin{equation}\label{soft_streak:weights_formula_using_generating_function:eq2}
        \Phi(x) = z^k - \frac{z^k}{k}\sum_{s=1}^{k-1} \frac{\qty(1-\omega_k^{-s})\qty(1-\omega_k^sz)}{1-x-\omega_k^sz}
    \end{equation}
    Applying \Cref{soft_streak:rational_function_coefficient_lemma} to \eqref{soft_streak:weights_formula_using_generating_function:eq2}, we get (for all $m\in\N$):
    \begin{equation*}
        w_m = -\frac{z^k}{k}\sum_{s=1}^{k-1} \qty(1-\omega_k^{-s})\qty(1-\omega_k^sz)^{-m} \qedhere
    \end{equation*}
\end{proof}
\begin{remark}
    Note that \eqref{soft_streak:weights_formula_using_generating_function:eq2} has a constant term $z^k$, as the numerator and denominator of \eqref{soft_streak:weights_formula_using_generating_function:eq1} have the same degree $k-1$. However, substituting $x=0$ into \eqref{soft_streak:weights_formula_using_generating_function:eq2} yields $\Phi(0)=0$, which makes sense as (by definition) $\Phi(x)$ has no constant term.
\end{remark}

\begin{proposition}\label{soft_streak:weights_general_formula}
    For all $m\in\N$, we have:
    \begin{equation}\label{soft_streak:weights_general_formula_eq}
        w_m = -\frac{z^k}{(1-z^k)^m}\sum_{r=0}^{(k-1)m} \psi_{k,r}\Bs(m,k,r)z^r
    \end{equation}
\end{proposition}

This follows directly by substituting \Cref{psi_k_r:identity_2} into \Cref{soft_streak:weights_formula_using_generating_function}.

\begin{proposition}\label{soft_streak:sum_of_weights}
    The weight $W(\Fs)$ of the set of all soft streaks is given by:
    \begin{equation*}
        W(\Fs) = 1 - nz - \frac{1}{\qty(1-z^k)^n}\sum_{r=0}^{(k-1)n} \psi_{k,r}\Bs(n,k,r)z^r
    \end{equation*}
\end{proposition}
\begin{proof}
    In this proof, we will use a generating function indexed by $n$ (the size of the alphabet), so for clarity, we will denote the set of forbidden words for each $n$ by $\Fs_n$ instead of $\Fs$. By \Cref{soft_streak:w_depends_only_on_1}, we have:
    \begin{equation*}
        W(\Fs_n) = \sum_{s=1}^n \binom{n+k-s-1}{k-1}w_s
    \end{equation*}
    Define $\Omega(x)=\sum_{n=1}^\infty W(\Fs_n)x^n$. Further manipulation yields:
    \begin{align*}
        \Omega(x)
        &= \sum_{n=1}^\infty \sum_{s=1}^n \binom{n+k-s-1}{k-1}w_sx^n \\
        &= \sum_{s=1}^\infty \sum_{n=s}^\infty \binom{n+k-s-1}{k-1}w_sx^n \tag{swap order of sums} \\
        &= \sum_{s=1}^\infty w_sx^s\sum_{n=s}^\infty \binom{n+k-s-1}{k-1}x^{n-s} \\
        &= \sum_{s=1}^\infty w_sx^s\sum_{q=0}^\infty \binom{q+k-1}{k-1}x^q \tag{$q=n-s$} \\
        &= \sum_{s=1}^\infty w_sx^s\sum_{q=0}^\infty \binom{q+k-1}{k-1}x^q \\
        &= \sum_{s=1}^\infty w_sx^s(1-x)^{-k} \\
        &= \frac{\Phi(x)}{(1-x)^k}
    \end{align*}
    Using \eqref{soft_streak:weights_formula_using_generating_function:eq1}, we get:
    \begin{equation}\label{soft_streak:sum_of_weights:eq1}
        \Omega(x) = \frac{-z^kx}{(1-x)^2\prod_{u=1}^{k-1} \qty(1-x-\omega_k^uz)}
    \end{equation}
    We now treat \eqref{soft_streak:sum_of_weights:eq1} as a rational function of $x$ and decompose it into partial fractions:
    \begin{equation*}
        \Omega(x) = -\frac{z}{(1-x)^2} + \frac{1+z}{1-x} - \frac{1}{k}\sum_{s=1}^{k-1} \frac{\qty(1-\omega_k^{-s})\qty(1-\omega_k^sz)}{1-x-\omega_k^sz}
    \end{equation*}
    We now expand each term as a power series in $x$ and apply \Cref{soft_streak:rational_function_coefficient_lemma} to the last term:
    \begin{align*}
        \Omega(x)
        &= -z\sum_{n=0}^\infty (n+1) x^n + (1+z)\sum_{n=0}^\infty x^n - \frac{1}{k}\sum_{n=0}^\infty \sum_{s=1}^{k-1} \qty(1-\omega_k^{-s})\qty(1-\omega_k^sz)^{-n}x^n \\
        &= (1-nz)\sum_{n=0}^\infty x^n - \frac{1}{k}\sum_{n=0}^\infty \sum_{s=1}^{k-1} \qty(1-\omega_k^{-s})\qty(1-\omega_k^sz)^{-n}x^n \\
        &= \sum_{n=0}^\infty \qty((1-nz) - \frac{1}{k}\sum_{s=1}^{k-1} \qty(1-\omega_k^{-s})\qty(1-\omega_k^sz)^{-n})x^n
    \end{align*}
    Thus:
    \begin{equation*}
        W(\Fs_n)
        = 1 - nz - \frac{1}{k}\sum_{s=1}^{k-1} \qty(1-\omega_k^{-s})\qty(1-\omega_k^sz)^{-n} \qedhere
    \end{equation*}
\end{proof}

\begin{corollary}
    The weight $W(\Fs)$ of the set of all soft streaks is given by:
    \begin{equation*}
        W(\Fs) = 1 - nz - \frac{1}{\qty(1-z^k)^n}\sum_{r=0}^{(k-1)n} \psi_{k,r}\Bs(n,k,r)z^r
    \end{equation*}
\end{corollary}

This follows directly by substituting \Cref{psi_k_r:identity_2} into \Cref{soft_streak:sum_of_weights}.\\

Finally, using the formula $f(z)=\frac{1}{1-nz-W(\Fs)}$ from the Goulden-Jackson method, we get the following expressions for the generating function $f(z)$:

\begin{theorem}\label{soft_streak:generating_function}
    The generating function $f(z)$ is given by:
    \begin{align}
        f(z)
        &= \frac{\qty(1-z^k)^n}{\Disp\sum_{r=0}^{(k-1)n} \psi_{k,r}\Bs(n,k,r)z^r}
        \label{soft_streak:generating_function:eq1} \\
        &= \frac{k}{\Disp\sum_{s=1}^{k-1} \qty(1-\omega_k^{-s})\qty(1-\omega_k^s z)^{-n}}
        \label{soft_streak:generating_function:eq2}
    \end{align}
\end{theorem}

This proves the conjecture in \cite[\S4]{sekhon}.

\begin{remark}[Special case $k=2$]
    If $k=2$, there is a shortcut: A word does not contain a soft streak of length $2$ if and only if it is strictly decreasing. The number of strictly decreasing words of length $s$ is simply the number of combinations of $\{1,2,\ldots,n\}$ of length $s$, i.e. $\binom{n}{s}$. Thus the generating function is given by:
    \begin{equation*}
        f(z)
        = \sum_{s=0}^\infty \binom{n}{s}z^s
        = (1+z)^n
    \end{equation*}
    This agrees with \eqref{soft_streak:generating_function:eq1} and \eqref{soft_streak:generating_function:eq2} as $\psi_{2,r}=(-1)^r$, $\omega_2=-1$ and $\Bs(n,2,r)=\binom{n}{r}$.\\
    
    Note that this series terminates at $s=n$. This makes sense, since if $s>n$, you cannot avoid a soft streak of length $2$ (the longest strictly decreasing word is $(n,n-1,\ldots,2,1)$ which has length $n$).
\end{remark}

\section{Application to Random Sampling}

In \cite[\S5]{sekhon}, we proved the following result:

\begin{theorem}[Expected number of draws to get a streak]\label{streak:expected_value}
    Suppose consecutive letters are drawn from $\{1,2,\ldots,n\}$ until a streak of $k$ is obtained. Then the expected number of letters drawn is given by:
    \begin{align}\label{streak:expected_value_eq}
        E(n,k)
        &= \frac{1}{\Disp\sum_{r=0}^n \psi_{k,r}\binom{n}{r}n^{-r}} \\
    \label{streak:expected_value_eq2}
        &= \frac{k}{\Disp\sum_{s=1}^{k-1} \qty(1-\omega_k^{-s})\qty(1+\frac{\omega_k^s}{n})^n}
    \end{align}
\end{theorem}

The analogous result for \emph{soft} streaks was stated and left as a conjecture:

\begin{theorem}[Expected number of draws to get a soft streak]\label{soft_streak:expected_value}
    Suppose consecutive letters are drawn from $\{1,2,\ldots,n\}$ until a soft streak of $k$ is obtained. Then the expected number of letters drawn is given by:
    \begin{align}\label{soft_streak:expected_value_eq}
        E_{\textsf{soft}}(n,k)
        &= \frac{\qty(1-n^{-k})^n}{\Disp\sum_{r=0}^{(k-1)n} \psi_{k,r}\Bs(n,k,r)n^{-r}} \\
    \label{soft_streak:expected_value_eq2}
        &= \frac{k}{\Disp\sum_{s=1}^{k-1} \qty(1-\omega_k^{-s})\qty(1-\frac{\omega_k^s}{n})^{-n}}
    \end{align}
\end{theorem}

This now follows directly by substituting $z=\frac{1}{n}$ into \eqref{soft_streak:generating_function:eq1} and \eqref{soft_streak:generating_function:eq2} respectively. Note that evaluating the generating function at $\frac{1}{n}$ is valid, as the resulting series is guaranteed to converge by \cite[Theorem A.6]{sekhon}.

\subsection{Continuous limit}

Taking the limit as $n\to\infty$ in both \eqref{streak:expected_value_eq2} and \eqref{soft_streak:expected_value_eq2} yield the same result:
\begin{align*}
    \limn E(n,k) = \limn E_{\textsf{soft}}(n,k) = \frac{k}{\Disp\sum_{s=1}^{k-1} e^{\omega_k^s}\qty(1-\omega_k^{-s})}
\end{align*}

This expression appears in \cite[\S6]{sekhon} and is denoted by $\mu_k$. It represents the expected number of draws from an absolutely continuous distribution to get a streak of $k$ (or a soft streak, the difference vanishes in the continuous case). We also conjectured (from numerical evidence) that $\mu_k=k!+(k-1)!+o(1)$ as $k\to\infty$.

\begin{lemma}\label{mu_k_asymptotic_lemma}
    \begin{equation*}
        \mu_k = \frac{1}{k\sumn\dfrac{n}{(nk+1)!}}
    \end{equation*}
\end{lemma}
\begin{proof}
    \begin{align*}
        \frac{k}{\mu_k}
        &= \sum_{s=1}^{k-1} e^{\omega_k^s}\qty(1-\omega_k^{-s}) \\
        &= \sum_{s=1}^{k-1} \sum_{t=0}^\infty \frac{\omega_k^{st}}{t!}\qty(1-\omega_k^{-s}) \\
        &= \sum_{t=0}^\infty \frac{1}{t!}\sum_{s=1}^{k-1} \omega_k^{st}\qty(1-\omega_k^{-s}) \\
        &= k\sum_{t=0}^\infty \frac{1}{t!}\psi_{k,t} \tag{by \cite[Lemma 1.2]{sekhon}}
    \end{align*}
    By the definition of $\psi_{k,t}$, we can split the outer sum into one sum containing the values of $t$ congruent to $0 \pmod{k}$, where we set $t=kn$, and another containing those congruent to $1 \pmod{k}$, where we set $t=kn+1$.
    \begin{align*}
        \frac{k}{\mu_k}
        &= k\qty(\sum_{n=0}^\infty \frac{1}{(kn)!} - \sum_{n=0}^\infty \frac{1}{(kn+1)!}) \\
        &= k\sum_{n=0}^\infty \qty(\frac{1}{(kn)!}-\frac{1}{(kn+1)!}) \\
        &= k\sum_{n=0}^\infty \frac{(kn+1)-1}{(kn+1)!} \\
        &= k^2\sum_{n=0}^\infty \frac{n}{(kn+1)!}
    \end{align*}
    Thus:
    \begin{equation*}
        \mu_k = \frac{1}{k\sum_{n=1}^\infty \dfrac{n}{(kn+1)!}} \qedhere
    \end{equation*}
\end{proof}

\begin{theorem}
    \begin{equation*}
        \mu_k = k!+(k-1)!+o(1) \qq{as} k\to\infty
    \end{equation*}
\end{theorem}
\begin{proof}
    By \Cref{mu_k_asymptotic_lemma}, we have:
    \begin{equation*}
        \frac{1}{k\mu_k}
        = \sumn \frac{n}{(kn+1)!}
        = \frac{1}{(k+1)!} + \frac{1}{(k+1)!}\underbrace{\sumn[2] \frac{n(k+1)!}{(kn+1)!}}_{\alpha_k}
    \end{equation*}
    This yields:
    \begin{equation*}
        \mu_k
        = \frac{1}{k\qty(\frac{1}{(k+1)!}+\frac{\alpha_k}{(k+1)!})}
        = \frac{(k+1)!}{k}\frac{1}{1+\alpha_k}
        = \qty(k!+(k-1)!)\frac{1}{1+\alpha_k}
    \end{equation*}
    It now suffices to show that $\limk k!\alpha_k=0$.
    \begin{equation*}
        \alpha_k
        = \sumn[2] \frac{n(k+1)!}{(kn+1)!}
        = \sumn[2] \frac{n}{(k+2)(k+3)\cdots (kn+1)}
        < \sumn[2] \frac{n}{k^{(n-1)k}}
    \end{equation*}
    Note that for all $a>1$, we have:
    \begin{equation*}
        \sumn[2] \frac{n}{a^{n-1}} = \frac{2a-1}{(a-1)^2} = \order{\frac{1}{a}} \qq{as} a\to\infty
    \end{equation*}
    Setting $a=k^k$, we get:
    \begin{align*}
        \alpha_k = \order{\frac{1}{k^k}} \qq{as} k\to\infty
    \end{align*}
    Finally, since $k!=o\qty(k^k)$ as $k\to\infty$, we have $\limk k!\alpha_k=0$.
\end{proof}

\subsection*{Acknowledgments}

I would like to thank Daniel Isaac Puignau Chacón and Diego Chicharro Gordo for their contributions.

\printbibliography

@article{sekhon,
    author      = {Sekhon, Senan},
    title       = {Counting words without strictly increasing subwords of fixed length},
    year        = {2025},
    url         = {https://arxiv.org/abs/2511.13287},
}

\end{document}